\documentclass[12pt]{amsart}

\newif \iffig
\figfalse

\usepackage{sfilip_presets,amsmath,amsthm}
\usepackage{graphicx}
\usepackage[stretch=10]{microtype}
\usepackage{titletoc}

\let\oldtocsection=\tocsection
\let\oldtocsubsection=\tocsubsection
\let\oldtocsubsubsection=\tocsubsubsection

\renewcommand{\tocsection}[2]{\hspace{0em}\oldtocsection{#1}{#2}}
\renewcommand{\tocsubsection}[2]{\hspace{1.75em}\oldtocsubsection{#1}{#2}}
\renewcommand{\tocsubsubsection}[2]{\hspace{2em}\oldtocsubsubsection{#1}{#2}}
\makeatletter
\renewcommand{\paragraph}{%
\@startsection {paragraph}{4}
{\z@} \z@ {-\fontdimen 2\font }\bfseries
}
\makeatother

\makeatletter
\renewcommand\section{\@startsection {section}{1}{\z@}%
                                   {-3.5ex \@plus -1ex \@minus -.2ex}%
                                   {2.3ex \@plus.2ex}%
                                   {\centering \normalfont\large\scshape}}% 
                                   
\renewcommand{\subsection}
{\@startsection{subsection}{2}{\z@}%
                                     {-3.25ex\@plus -1ex \@minus -.2ex}%
                                     {1.5ex \@plus .2ex}%
                                     {\normalfont \scshape}}% from \large
\makeatother

\newif\ifdebug
\debugfalse

\ifdebug
  \usepackage[status=draft,author=]{fixme}
  \fxsetup{inline, marginclue,theme=color}
\else
  \usepackage[status=final, author=]{fixme}
\fi
\usepackage{hyperref}
\usepackage{color}
\definecolor{darkred}{rgb}{0.4,0,0}
\definecolor{darkgreen}{rgb}{0,0.5,0}
\definecolor{darkblue}{rgb}{0,0,0.4}

\hypersetup{
    pdftitle={},	% title
    pdfauthor={Simion Filip},	% author
    pdfsubject={math},		% subject of the document
    pdfkeywords={keyword1},	% list of keywords
    pdfnewwindow=true,		% links in new window
    colorlinks=true,		% false: boxed links; true: colored links
    linkcolor=darkblue,		% color of internal links (change box color with linkbordercolor)
    citecolor=darkred,		% color of links to bibliography
    filecolor=darkblue,		% color of file links
    urlcolor=darkblue,		% color of external links
    pdfborder={0 0 0},
    breaklinks=true
}
\usepackage{aliascnt}

\numberwithin{equation}{subsection}

\def\subsek~{\S{}}

\newtheoremstyle{mytheoremstyle} % name
    {5pt}                    % Space above
    {5pt}                    % Space below
    {\itshape}                   % Body font
    {}                           % Indent amount
    {\bfseries}                   % Theorem head font
    {}                          % Punctuation after theorem head
    {.5em}                       % Space after theorem head
    {}  % Theorem head spec (can be left empty, meaning ‘normal’)

\theoremstyle{mytheoremstyle}
    
\newtheorem{theorem}{Theorem}[section]

\newaliascnt{lemma}{theorem}  
  
\aliascntresetthe{lemma}

\newaliascnt{proposition}{theorem}
\newtheorem{proposition}[proposition]{Proposition}
\aliascntresetthe{proposition}

\newaliascnt{corollary}{theorem}  
  
\aliascntresetthe{corollary}

\newaliascnt{exercise}{theorem}  
  
\aliascntresetthe{exercise}

\newaliascnt{definition}{theorem}  
\newtheorem{definition}[definition]{Definition}  
\aliascntresetthe{definition}

\newaliascnt{remark}{theorem}  
\newtheorem{remark}[remark]{Remark}  
\aliascntresetthe{remark}

\newaliascnt{example}{theorem}  
\newtheorem{example}[example]{Example}  
\aliascntresetthe{example}

\newcommand{\bmu}{\mbox{$\raisebox{-0.59ex}
  {$l$}\hspace{-0.18em}\mu\hspace{-0.88em}\raisebox{-0.98ex}{\scalebox{2}
  {$\color{white}.$}}\hspace{-0.416em}\raisebox{+0.88ex}
  {$\color{white}.$}\hspace{0.46em}$}{}}

\usepackage{etoolbox}
\usepackage{needspace}
\AtBeginEnvironment{theorem}{\Needspace{5\baselineskip}}
\title[Families of K3 surfaces and Lyapunov exponents]{Families of K3 surfaces and Lyapunov~exponents}
\thanks{{Revised \textsc{\today}}}
\author{{ Simion Filip}}

\usepackage{etoolbox}
\makeatletter
\patchcmd{\@maketitle}
  {\ifx\@empty\@dedicatory}
  {\ifx\@empty\@date \else {\vskip3ex \centering\footnotesize\@date\par\vskip1ex}\fi
   \ifx\@empty\@dedicatory}
  {}{}
\patchcmd{\@adminfootnotes}
  {\ifx\@empty\@date\else \@footnotetext{\@setdate}\fi}
  {}{}{}
\makeatother

\date{4 December 2014}

\address{
\parbox{0.5\textwidth}{
Department of Mathematics\\
University of Chicago\\
Chicago IL, 60615\\}
	}
\email{{sfilip@math.uchicago.edu}}

\begin{document}

\begin{abstract}
Consider a family of K3 surfaces over a hyperbolic curve (i.e. Riemann surface).
Their second cohomology groups form a local system, and we show that its top Lyapunov exponent is a rational number.
One proof uses the Kuga-Satake construction, which reduces the question to Hodge structures of weight $1$.
A second proof uses integration by parts.
The case of maximal Lyapunov exponent corresponds to modular families, given by the Kummer construction on a product of isogenous elliptic curves.
\end{abstract}

\maketitle

%		TOC
\tableofcontents

%====================================================

%		Things to be fixed
\ifdebug
  \listoffixmes
\fi
%====================================================

\section{Introduction}

Let $C$ be a hyperbolic Riemann surface and let $H\to C$ be a local system over $C$; this is the same as a linear representation of the fundamental group of $C$.
Pick a random point $x$ on $C$ and a random direction $\theta$ at $x$, and let $\gamma_T$ be the hyperbolic geodesic of length $T$ starting at $x$, in direction $\theta$, and of length $T$.
Connect the endpoint of $\gamma_T$ to $x$ to get a closed loop on $C$, and thus a monodromy matrix $M_{\gamma_T}$.
As $T$ grows, what can one say about the eigenvalues of $M_{\gamma_T}$?

The answer is given by the Oseledets theorem: there exist numbers
$$
\lambda_1\geq \cdots \geq \lambda_n
$$
such that the eigenvalues of $M_{\gamma_T}$ grow like $e^{\lambda_i T}$, as long as the starting point $x$ and the direction $\theta$ are chosen randomly (with respect to the natural measure).

The above numbers are called Lyapunov exponents and are not directly computable - they arise as the limit of a subadditive sequence.
Zorich \cite{Zorich_leaves} discovered experimentally, and then Kontsevich \cite{Kontsevich} explained that when the local system comes from a variation of Hodge structure of weight $1$, the sum of positive exponents is rational.
Forni \cite{Forni} developed and detailed this analysis further.
%Note that in this case, the exponents are symmetric w.r.t. zero, i.e. if $\lambda$ occurs, so does $-\lambda$ (with the same multiplicity).

The method has been extended by Eskin-Kontsevich-Zorich \cite{EKZ} to apply to the more general setting of \Teichmuller dynamics, but still in the context of weight $1$ variations of Hodge structure.
More recently, the work of Yu \cite{FeiYu} makes a conjecture about a relationship between the Lyapunov exponents and the Harder-Narasimhan filtration of the variation of Hodge structures.
In analogy with the situation in crystalline cohomology, it is conjectured that a polygon constructed from Lyapunov exponents lies above a polygon constructed from the Harder-Narasimhan filtration.

An extension of the method, when the base is not a Riemann surface but rather a ball quotient, was done by Kappes and M\"{o}ller \cite{Kappes_Moller}.
It applies to weight $1$ variations and allows them to distinguish commensurability classes of lattices in $\SU(n,1)$.
McMullen \cite{McMullen} related their computations with volumes of complex-hyperbolic manifolds.

This paper investigates the case of weight $2$ variations of Hodge structure of K3 type.
In this case, the Hodge numbers are $(1,n,1)$ and the monodromy preserves an indefinite quadratic form.
The Lyapunov exponents then have to be of the form
$$
\lambda_1\geq \lambda_2 \geq 0 \geq\cdots\geq 0 \geq -\lambda_2 \geq -\lambda_1
$$
The main result of the paper (see \autoref{thm:top_exp}) is that $\lambda_1$ is rational.

\begin{theorem}
 Let $H\to C\setminus S$ be a polarized variation of Hodge structure of K3 type over a compact Riemann surface $C$ with finitely many punctures $S$.
 
 Then the top Lyapunov exponent of the local system for the geodesic flow is rational and given by the formula
 \begin{align}
 \label{eqn:formula}
  \lambda_1 = \frac 12 \frac{\deg H^{2,0}}{\deg K_C(\log S) }
 \end{align} 
 Here $H^{2,0}\subset H_\bC$ is the holomorphic line bundle describing the Hodge structure and $K_C(\log S)$ is the logarithmic canonical bundle of $(C,S)$.
 
 We also have the apriori bound
 $$\lambda_1 \leq \frac 12$$
 The equality case corresponds to special (i.e. modular) families and is discussed in \autoref{subsec:max_exp}.
\end{theorem}

The standard examples come from families of K3 surfaces over a complex curve (i.e. Riemann surface).
The local system $H$ consists of the second cohomology group of the surfaces in the family.
Note that a K3 surface is complex $2$-dimensional, so has $4$ real dimensions.
\autoref{sec:examples} works out a number of classical families.

\begin{remark}
 In \Teichmuller dynamics, the variation always has a rank $2$ direct summand whose top exponent\footnote{With the normalizations of this paper, one should say that the top exponent~is~$\frac 14$} is $1$, and the other one necessarily $-1$.
 This summand gives the uniformization of the hyperbolic metric and accounts for maximal growth rate.
 
 This is not the case for families of K3 surfaces.
 The families of K3s with largest top exponent are described in \autoref{subsec:max_exp}.
\end{remark}

\paragraph{Normalizations.}
The formula in \autoref{eqn:formula} assumes that the hyperbolic metric has curvature $-1$ and the geodesic flow has unit speed.
This is in contrast with \cite{EKZ} where the curvature is $-4$.

\paragraph{Outline.}
\autoref{sec:Prelims_Hg} reviews basic facts from Hodge theory.
It also contains a discussion of the Kuga-Satake construction, following Deligne \cite{Deligne_K3}.

\autoref{sec:Prelims_Dyn} reviews the background from dynamics.
It also contains some calculations with spin representations, necessary for applying the Kuga-Satake construction.

\autoref{sec:top_exp} computes the top Lyapunov exponent in two different ways.
The first method uses the Kuga-Satake construction to reduce the question to weight $1$ variations.
The second method is direct and uses integration by parts, similar to the weight $1$ case.

\autoref{sec:examples} contains a number of classical examples to which the main theorem applies.
The case of maximal exponent $\lambda_1=\frac 12$ and its geometric meaning is discussed in \autoref{subsec:max_exp}.

\paragraph{Some remarks.}
The connection with dynamics on individual K3 surfaces is not clear at the moment.
For this topic, the work of Cantat \cite{Cantat} and McMullen \cite{McMullen_Siegel} can serve as an introduction.

\paragraph{Acknowledgments.}
I am grateful to Alex Eskin, Maxim Kontsevich, Martin M\"{o}ller, and Anton Zorich for conversations on this subject.
I am also grateful to Martin M\"{o}ller for doing a numerical experiment which confirmed the rationality result of this paper.
I have benefited also from conversations with Daniel Huybrechts, who in particular pointed out \cite{Maulik}.

\section{Preliminaries from Hodge theory}
\label{sec:Prelims_Hg}

This section contains some background on Hodge theory and the Kuga-Satake construction.
Some basic definitions are in \autoref{subsec:Hg_basics}, followed by a description of the Kuga-Satake construction in \autoref{subsec:KS}.
The Kuga-Satake construction in families is discussed in \autoref{subsec:KS_fam}.

\subsection{Hodge structures and K3 surfaces}
\label{subsec:Hg_basics}

This section recalls the definition of Hodge structures and discusses some examples.
The definitions follow Deligne \cite{Deligne_travaux}, but see also \cite{PerDom} for a leisurely introduction.

Denote by $\bS$ the real algebraic group whose $\bR$-points are $\bC^\times$.
The real-algebraic structure on $\bC^\times$ is from its embedding as subgroup of $2\times 2$ real matrices.
It arises via the natural action on $\bC=\bR^2$.

\begin{definition}
\label{def:Hg_str}
 A \emph{Hodge structure of weight $w$} on a real vector space $H_\bR$ is a decomposition of its complexification $H_\bC$ as
 $$
 H_\bC = \bigoplus_{p+q=w} H^{p,q}
 $$
 such that $H^{p,q}=\conj{H^{q,p}}$.
 The \emph{Hodge filtration} on $H_\bC$ is defined by
 $$
 F^p H_\bC = \bigoplus_{i\geq p} H^{i,q}
 $$
 It is a decreasing filtration:
 $$
 H_\bC \supseteq \cdots \supseteq F^{p-1}\supseteq F^p \supseteq \cdots \supseteq \{0\}
 $$
 The filtration also determines the Hodge decomposition via
 $$
 H^{p,q}:=F^p\cap \conj{F^{q}}
 $$
 Let any $z\in \bC^\times$ act on $H^{p,q}$ by $z^p\conj{z}^q$.
 Since complex conjugation exchanges the $(p,q)$ and $(q,p)$ components of $H_\bC$ this action descends to a homomorphism $h:\bS \to \GL(H_\bR)$ of real algebraic groups.
  
 Equivalently, a Hodge structure on $H_\bR$ is the same as a representation $h:\bS \to \GL(H_\bR)$.
 The weight is $w$ if elements $x\in\bR^\times\subset \bS$ act by $h(x)=x^w$.
 
 The \emph{Weil operator} $C$ is defined as $C:=h(\sqrt{-1})$ and it is a real operator.
 On $H^{p,q}$ it acts by $\sqrt{-1}^{p-q}$.
\end{definition}

\begin{definition}
 A polarization of a Hodge structure on $H_\bR$ of weight $w$ is a bilinear form $I(-,-)$ on $H_\bR$ such that:
 \begin{itemize}
  \item It is non-degenerate and $(-1)^w$-symmetric (skew-symmetric for odd weight, symmetric for even weight).
  \item The orthogonal of $F^p$ with respect to $I$ is $F^{w+1-p}$.
  \item The bilinear form $Q(x,y):=I(Cx,\conj{y})$ on $H_\bC$ is hermitian and positive-definite.
 \end{itemize}
 Equivalently, the representation $h:\bS\to \GL(H_\bR)$ factors through $\GL(H_\bR,I)$, the linear transformations that preserve $I$ up to scaling.
 The form $Q(x,y):=I(Cx,\conj{y})$ must again be hermitian and positive-definite.
 
 We call $I(-,-)$ the indefinite form and $Q(-,-)$ the definite form.
\end{definition}

\begin{definition}
 An \emph{integral Hodge structure} of weight $w$ is a free $\bZ$-module $H_\bZ$ with a real Hodge structure of weight $w$ on $H_\bR:=H_\bZ\otimes_\bZ \bR$.
 A polarization is a an integer-valued bilinear form $I$ on $H_\bZ$ which induces a polarization on $H_\bR$.
\end{definition}

\begin{example} See \cite[Ch. 0]{GH} for an introduction to these ideas.
 \begin{enumerate}
 \item The primitive cohomology of a projective algebraic variety $X$ carries a Hodge structure of corresponding weight.
 Recall that in the projective case, we have the class of some hyperplane section $[D]\in H^2(X;\bZ)\cap H^{1,1}$.
 A class in $H^k(X;\bC)$ is \emph{primitive} if the cup product with $[D]^{n-k+1}$ is trivial (here $n=\dim_\bC X$).
 The notion of primitive class depends on the choice of $[D]$, i.e. of projective embedding.
 
  \item If $A$ is an abelian variety, then $H^1(A;\bZ)$ carries a polarisable weight $1$ Hodge structure.
  In this case
  $$
  H^1(A;\bC) = H^{1,0}\oplus H^{0,1}
  $$
  Given the class of an ample divisor $[D]\in H^2(A;\bZ)$ a polarization is given by $I(x,y)=x\cdot y \cdot [D]^n$, where $n=\dim_\bC A - 1$.
  \end{enumerate}
\end{example}

A comprehensive introduction to the geometry of K3 surfaces is in the collected seminar notes \cite{K3}.
A more recent account is in the upcoming book of Huybrechts \cite{Huybrechts_K3}.

\begin{example}
  If $X$ is a projective K3 surface, then its primitive second cohomology $H^2(X;\bQ)_0$ carries a polarized Hodge structure of weight $2$.
  Recall that a $K3$ surface is a complex surface, so a real $4$-manifold.
  In particular, cup product defines a symmetric (even, unimodular) bilinear pairing on $H^2$.
  
  Since $X$ is projective, it has a class $[D]\in H^2(X;\bZ)\cap H^{1,1}$ corresponding to a hyperplane section (as well as to a \Kahler form).
  Primitive cohomology is the orthogonal complement of $[D]$ for cup product, denoted $H^2(X)_0$.
  Its Hodge decomposition over $\bC$ is  
  $$
  H^2(X;\bC)_0 = H^{2,0} \oplus H^{1,1} \oplus H^{0,2}
  $$
  where $\dim_\bC H^{2,0}=1$ and $\dim_\bC H^{1,1}=19$. 
  Cup product gives a natural polarization, i.e. a quadratic form on $H^2(X;\bR)_0$ of signature $(2+,19-)$.
\end{example}

\begin{definition}
 A Hodge structure of weight $2$ is \emph{of K3 type} if 
 $$
 \dim_\bC H^{2,0}=\dim_\bC H^{0,2}=1
 $$
 and no other weights except $H^{1,1}$ appear in the Hodge decomposition.
 If it is polarized, then the quadratic form has signature $(2+,n-)$.
\end{definition}

\begin{remark}
 Let $H_\bZ$ be a Hodge structure of weight $1$, with $H_\bZ$ of rank $4$ over $\bZ$.
 Then the second exterior power $\Wedge^2 H_\bZ$ naturally carries a Hodge structure of K3 type.
 
 When applied to local systems, the top exponent of $\Wedge^2 H$ will equal the sum of the positive exponents of $H$.
 This recovers the classical rationality result for the sum of exponents in this particular case.
\end{remark}

\subsection{The Kuga-Satake construction}
\label{subsec:KS}

Let $(H,I)$ be a polarized Hodge structure of K3 type.
As explained in \autoref{subsec:Hg_basics}, this gives a representation 
$$
h:\bS\to \GL(H,I)
$$
The group $\GL(H,I)$ in an orthogonal group (including scalings) and acts in the standard representation on $H$.
The representation $h$ lifts to the Clifford group $\tilde{h}:\bS\to \CSpin(H,I)$.
A spin representation of $\CSpin(H,I)$ on a real vector space endows it with a Hodge structure, via the map $\tilde{h}$.
A computation shows it is of weight $1$.

Below are the details of this construction, introduced by Kuga and Satake \cite{KS}.
A detailed exposition is available in the notes of Huybrechts \cite[Ch. 4]{Huybrechts_K3} (see also \cite[Sect. 3]{Deligne_K3})

\paragraph{Clifford algebras.}
Consider an integral Hodge structure $(H_\bZ,I)$ of K3 type.
The Clifford algebra can be defined with any coefficients (i.e. $\bZ,\bR$ or $\bC$) and only depends on the quadratic form $I$, not the Hodge structure.
Using generators and relations, it is
$$
\Cl(H) :=\left.\left( \bigoplus_{n\geq 0} H^{\otimes n} \right)\middle/ \{ v\otimes v = I(v,v)\}\right.
$$
It decomposes canonically into elements with an odd and even number of terms $\Cl(H) = \Cl_+ H \oplus \Cl_- H$.
The algebra also carries an anti-involution (called transposition) defined by
$$
(v_1\cdots v_n)^t := v_n\cdots v_1
$$
\paragraph{Clifford and spin groups.}
The Clifford group is the subset of invertible even elements which preserve $H\subset \Cl(H)$ under conjugation:
$$
\CSpin(H):=\{x\in \Cl^\times_+ \vert xHx^{-1}=H\}
$$
The conjugation action on $H$ defines the standard representation of $\CSpin$, with image the orthgonal group of $H$:
$$
\rho_{std}:\CSpin(H)\to \operatorname{O}(H)
$$
The kernel of this representation consists of scalars.

The left action of $\CSpin$ on $\Cl_+$ defines another representation
\begin{align*}
 \rho_{spin}:\CSpin(H) &\to \GL(\Cl_+)\\
 \rho_{spin}(x)v &= x\cdot v 
\end{align*}
This action is not irreducible, and in fact decomposes into several copies of a spin representation $W$.
Depending on the parity of $\dim H$, the representation $W$ might decompose further.
A more detailed analysis is in \autoref{subsec:spin_reps}.

The spin group is the subgroup of unit norm elements:
$$
\Spin(H):= \{x\in \CSpin\vert x\cdot x^t = 1\}
$$
The representations $\rho_{std}$ and $\rho_{spin}$ restrict to the spin group.

\paragraph{Complex structure.}
Recall that $H$ carries a Hodge structure of K3 type, i.e. a decomposition
$$
H_\bC=H^{2,0}\oplus H^{1,1} \oplus H^{0,2}
$$
Define the real $2$-dimensional plane $P=H_\bR \cap \left(H^{2,0}\oplus H^{0,2}\right)$, on which $I$ is positive-definite.

Choose a positively oriented orthonormal basis $e_1,e_2$ of $P$.
The orientation is fixed by requiring that $\omega:=e_1+\sqrt{-1}e_2$ spans $H^{2,0}$.
Note that with the current sign conventions we have $e_1^2=e_2^2=1$ inside $\Cl(H)$.

Define the element $J_P := e_1 e_2\in \CSpin(H_\bR)\subset \Cl_+(H_\bR)$, and note it is independent of the choice of orthonormal basis $\{e_1,e_2\}$.
It satisfies the identity
$$
J_P^2 = e_1 e_2 e_1 e_2 = - e_1 e_1 e_2 e_2 = -1
$$
In particular, $J_P^{-1}=-J_P$.

Using the representation $\rho_{std}$ of $\CSpin$ on $H$, the element $\rho_{std}(J_P)$ induces the action of the Weil operator in the Hodge decomposition.
Indeed, on $H^{1,1}$ the action is trivial since $J_P$ will commute with any element there.
On $P=\left(H^{2,0}\oplus H^{0,2}\right)\cap H_\bR$ the action is of rotation by $\pi/2$, so induces the usual complex structure.

Using now the representation $\rho_{spin}$, the action of $\rho_{spin}(J_P)$ on $\Cl_+(H_\bR)$ is by left multiplication.
It induces a new complex structure, and so can be viewed as a Hodge structure of weight $1$.

\newcommand{\KS}{{\operatorname{KS}}}

\begin{definition}
 The \emph{Kuga-Satake Hodge structure of weight $1$} associated to $(H,I)$ is the $\bZ$-module $\Cl_+(H_\bZ)$, with complex structure on $\Cl_+(H_\bR)$ given by the operator $J_P$ defined above.
 It is denoted $\KS(H,I)$.
\end{definition}

\begin{proposition}
\label{prop:Cl10}
 With notation as above, the space $\Cl^{1,0}$ consists of all multiples of $\omega:=e_1 +\sqrt{-1} e_2$.
 It also consists of elements $\alpha\in \Cl(H_\bC)$ such that $\omega\cdot \alpha =0$.
\end{proposition}
\begin{proof}
 Inside the Clifford algebra we have the identities (where $\conj{\omega}=e_1-\sqrt{-1}e_2$)
 \begin{align*}
 J_P\cdot \omega &= \sqrt{-1}\omega\\ 
 \omega\cdot \omega& = 0\\
 \omega\cdot \conj{\omega} & = 2(1-\sqrt{-1}J_P)
 \end{align*} 
 The first calculation is
 $$
 J_P \cdot \omega = e_1 e_2 (e_1+\sqrt{-1}e_2) = -(e_1 e_1) e_2 + \sqrt{-1} e_1 = \sqrt{-1}\omega
 $$
 and the other identities follow similarly.
 In particular, for any $\alpha$ we have $J_P(\omega\alpha) = \sqrt{-1}\omega \alpha$.
 So multiples of $\omega$ are inside $\Cl^{1,0}$.
 
 Next, multiplication by $\omega\cdot \conj{\omega}$ acts by $+4$ on $\Cl^{1,0}$ and by $0$ on $\Cl^{0,1}$.
 So for any $\alpha$ we see that $\frac 14 \omega \conj{\omega}\alpha$ is the $(1,0)$ component of $\alpha$.
 In particular, any element of $\Cl^{1,0}$ is a multiple of $\omega$.
 
 Finally, from $\omega\cdot \omega =0$ we find that multiples of $\omega$ are at most half the dimension of $\Cl$.
 Because $\Cl^{1,0}$ agrees with the multiples of $\omega$, we find that it must also equal the set of elements annihilated by $\omega$.
\end{proof}

\begin{remark}
 The odd part of the Clifford algebra $\Cl_-(H)$ also carries a weight $1$ Hodge structure, non-canonically isogenous to $\Cl_+(H)$.
 Choose $v\in H_\bZ$ such that $I(v,v)\neq 0$.
 Then multiplication by $v$ on the right in $\Cl(H)$ exchanges the odd and even components, and commutes with left multiplication by $J_P$.
 Multiplying by $v$ on the right twice is just scalar multiplication by $I(v,v)$.
 Therefore, any choice of $v$ will provide an isogeny.
\end{remark}

\paragraph{Polarizations.}
The complex torus $\KS(H,I)$ constructed above is not naturally polarized.
However, a polarization can be constructed as follows.

For $\alpha\in \Cl(H)$ define $\tr(\alpha)$ to be the trace of the operator of left multiplication by $\alpha$ in $\Cl(H)$.
To define the polarization, pick orthogonal vectors $f_1,f_2\in H_\bZ$ such that $I(f_i,f_i)>0$.

\begin{proposition}
\label{prop:polarizations}
 Recall that $x\mapsto x^t$ denotes the anti-involution of the Clifford algebra.
 On $\Cl_+(H)$ define the bilinear form
 $$
 I_\KS (x,y):=\pm \tr (f_1 f_2 \cdot x^t \cdot y)
 $$ 
 This gives an integral polarization of the weight $1$ Kuga-Satake Hodge structure on $\Cl_+(H)$.
 
 The sign in the definition is determined as follows.
 The possible polarizations (up to sign, and not necessarily integral) form two connected components.
 Choosing $f_1=\pm e_1, f_2=e_2$ in the definition above gives representatives in each component, and the sign is chosen to make $I_\KS$ positive.
 Here $e_1,e_2$ are the vectors used to define $J_P=e_1 e_2$ earlier.
\end{proposition}

\begin{proof}
 See \cite[Prop. 5.9]{vG_KS} and the discussion before it.
\end{proof}

\subsection{The Kuga-Satake construction in families}
\label{subsec:KS_fam}

This section explains how to extend the above construction to variations of Hodge structure.

\paragraph{Variations of Hodge structure.}
For a general introduction to variations of Hodge structure, see \cite[Ch. III]{PerDom}.
First, recall some relevant concepts.

\begin{definition}
\label{def:VHS}
 A \emph{variation of Hodge structure $H$ of weight $w$} over a complex manifold $B$ consists of the following data:
 \begin{itemize}
  \item A local system $H_\bZ$ of free $\bZ$-modules over $B$.
  \item Holomorphic subbundles $F^p\subseteq H_\bC$ of the complexified local system, inducing a Hodge structure of weight $w$ on each fiber of the bundle (see \autoref{def:Hg_str}).
  \item If $\nabla$ denotes the flat connection induced from the local system (called the Gauss-Manin connection), we require that
  $$
  \nabla F^p \subseteq F^{p-1}
  $$
  This is the Griffiths transversality condition.
 \end{itemize}
 The variation is \emph{polarized} if in addition we have on $H_\bZ$ a bilinear pairing $I(-,-)$, flat for the Gauss-Manin connection and inducing a polarization on each fiber of the bundle.
 Note that the positive-definite form $Q(-,-)$ induced from $I$ is typically not flat.
\end{definition}

\begin{remark}
 For variations of Hodge structure of weight $1$, as well as those of K3 type, the Griffiths transversality condition is automatically satisfied.
 In particular, the period domains (described below) are hermitian symmetric.
\end{remark}

\begin{remark}
 \label{rmk:VHS_bdd}
 Let $H\to \bD$ be a variation of Hodge structure on the unit disk.
 According to a result of Griffiths and Schmid \cite[Sect. 9]{GrSch} (see \cite[Cor. 13.4.2]{PerDom}), the classifying map from $\bD$ to the period domain (described below) is contracting.
 The period domain has a natural metric, and $\bD$ is equipped with the hyperbolic metric.
 
 This implies that the boundedness conditions required for the Ose\-ledets theorem are automatically satisfied.
 Namely, the integrand appearing in \autoref{thm:Oseledets} (see \autoref{eq:Os_int}) is uniformly bounded.
\end{remark}

\paragraph{Period domains.}
We now describe ``moduli spaces" of Hodge structures.
Due to the remark above, only in two exceptional cases (which are the ones we consider), these are true moduli spaces.
In other words, \emph{any} holomorphic map from a complex manifold to the moduli space gives a variation of Hodge structure.

\begin{definition}
 Let $H_\bR$ be a real vector space equipped with a non-degenerate $(-1)^w$-symmetric bilinear form $I$.
 Fix numbers $\{h^{p,q}\}_{p+q=w}$ with $h^{p,q}=h^{q,p}$ (called Hodge numbers) and let $f^p = \sum_{i\geq p} h^{i,w-i}$.
 
 The \emph{period domain $X$} of Hodge structures with the above numerical data is defined as follows.
 Consider inside the flag variety the closed subset
 \begin{align*}
 \check{D}:=&\left\lbrace\text{flags } F^p\subseteq F^{p-1}\subseteq \cdots H_\bC \text{ with } \dim_\bC F^p=f^p\right.\\
 &\left.\text{and } F^p \text{ is }I\text{-orthogonal to } F^{w+1-p}
 \right\rbrace
 \end{align*}
 The filtration $F^\bullet$ determines a decomposition $H^{p,q}:=F^p\cap \conj{F^q}$.
 
 The period domain $X\subset \check{D}$ is the open subset for which 
 $$
 (-1)^{p-q}I(\alpha,\conj{\alpha})>0 \text{\hskip 0.2in for all }\alpha\in H^{p,q}$$
\end{definition}

\paragraph{Period domains as homogeneous spaces.}
From the above definition, period domains carry natural transitive group actions.
The (complex) manifold $\check{D}$ can be written as $G(\bC)/P$ for $G$ an algebraic group and $P$ a parabolic subgroup.
Moreover, one can choose $G$ and $P$ such that the period domain is $X=G(\bR)/K$ where $K$ is a compact subgroup of $G(\bR)$ and moreover $K=G(\bR)\cap P$.

For Hodge numbers $(g,g)$ the period domain is the Siegel upper half-space 
$$
\frakh_g:=\Sp_{2g}(\bR)/U_g
$$
For Hodge structures of K3 type, i.e. Hodge numbers $(1,n,1)$ the period domain is 
$$
\Omega_n:=\SO_{2,n}(\bR)/ \SO_2(\bR)\times \SO_n(\bR)
$$
More concretely, let $H_\bR$ be a real vector space with a symmetric bilinear form $I$ of signature $(2+,n-)$.
The flag variety $\check{D}_{\Omega_n}$ which contains $\Omega_n$ is a quadric hypersurface in $\bP(H_\bC)$:
$$
\check{D}_{\Omega_n}:=\left\lbrace [v]\in \bP(H_\bC) \vert I(v,v)=0\right\rbrace
$$
It parametrizes lines corresponding to $H^{2,0}$; once this is known, $H^{2,0}\oplus H^{1,1}$ is determined as the $I$-orthogonal of $H^{2,0}$.
The period domain $\Omega_n$ is the open subset of those $[v]$ for which $I(v,\conj{v})>0$.

The hermitian symmetric domain $\Omega_n$ caries a natural variation of Hodge structure of K3 type, by construction.
The local system is given by $H_\bC$, which is just constant over the base.
The line subbundle $H^{2,0}$ varies holomorphically.

Giving the $1$-dimensional complex subspace $H^{2,0}\subset H_\bC$ is equivalent to giving the $2$-dimensional real subspace $P:=\left(H^{2,0}\oplus H^{0,2}\right)\cap H_\bR$.
This gives a (left) action of the orthogonal group $\SO_{2,n}(H_\bR)$ on $\Omega_n$.
The action extends equivariantly to the bundles $H^{2,0}\subset H_\bC$.
On $H_\bC$ the action is via the standard representation.

We shall need two descriptions of $\Omega_n$.
Let $G:=\SO_{2,n}$ and $G_1:=\Spin_{2,n}$ be its spin double cover, with $K$ and $K_1$ their maximal compacts.
Then $\Omega_n = G(\bR)/K=G_1(\bR)/K_1$.

\paragraph{Automorphic vector bundles.}
A more detailed discussion of the next topic (in the context of Shimura varieties) is given by Milne in \cite[Sect. III]{Milne}.

Consider a space equipped with a group action $G\curvearrowright X$ and a vector bundle $\cV\to X$ with the action of $G$ extending to $\cV$.
If $X=G/H$ for some subgroup $H\subseteq G$ and $\rho:H\to \GL(V)$ is a representation of $H$, an equivariant vector bundle on $X$ is defined by
$$
\cV_\rho:=G\times_H V = \left. G\times V \middle/ \left\lbrace (g,v)\sim (gh^{-1},\rho(h)v)\right\rbrace \right.
$$
We have a natural map $\cV_\rho\to G/H$ and an equivariant $G$-action on the left.
The vector bundle associated to a representation $\rho$ is denoted $\cV_\rho$.

\begin{remark}
 In the construction above, if the representation $\rho$ extends to $G$, then we have a $G$-equivariant isomorphism
 $$
 \cV_\rho \cong \left(G/H\right)\times V
 $$
 On the right-hand side $G$ acts diagonally.
\end{remark}

\paragraph{Automorphic vector bundles on $\Omega_n$.}
Fix a free $\bZ$-module $H_\bZ$ with a symmetric billinear form $I(-,-)$ of signature $(2+,n-)$ on $H_\bR$.
We also fix a Hodge structure of K3 type on $H$, i.e. a decomposition
$$
H_\bC=H^{2,0}\oplus H^{1,1}\oplus H^{0,2}
$$
With the notation from previous paragraphs, we have the groups $G=\SO(H,I)$ and $G_1=\Spin(H,I)$.
The Hodge structure on $H_\bC$ gives parabolics $P\subset G,P_1\subset G_1$ fixing $H^{2,0}$.

The representations relevant for the Kuga-Satake construction are:
\begin{align*}
\rho_{spin}&:G_1 \to \Aut(\Cl_+(H))\\
\rho_{std}&:G_1 \to \Aut(H)\\
\rho_{tw}&:P_1\to \Aut(H^{2,0})
\end{align*}
This leads to equivariant vector bundles $\cV_{\rho_{spin}},\cV_{\rho_{std}}, \cV_{\rho_{tw}}$.
These are defined on the bigger space $G_1(\bC)/P_1=\check{D}$ containing the period domain.

We also have an inclusion $\cV_{\rho_{tw}}\subset \cV_{\rho_{std}}$ compatible with the action of $G_1$.
Over $\Omega_n$ these give the universal variation of Hodge structure of K3 type, with $\cV_{\rho_{tw}}$ serving as $H^{2,0}$.

\paragraph{The weight $1$ automorphic vector bundle on $\Omega_n$.}
The next step describes the automorphic subbundle $\Cl^{1,0}_+\subset \cV_{\rho_{spin}}$ giving the weight $1$ variation of Hodge structure.
After choosing two integral vectors $f_1, f_2$ as in \autoref{prop:polarizations}, the bundle $\cV_{\rho_{spin}}$ acquires an equivariant polarization.

Next, fix some $v_0\in H_\bR$ such that $I(v_0,v_0)\neq 0$.
Clifford multiplication gives a $G_1$-equivariant map
\begin{center}
\begin{tabular}{rccc}
 $\phi:$	&$\cV_{\rho_{tw}}\otimes \cV_{\rho_{spin}}$ &$\to$	& $\cV_{\rho_{spin}}$\\
  &$  \omega\otimes \alpha $			     &$\mapsto$& $\omega\cdot \alpha \cdot v_0$
\end{tabular}
\end{center}
Equivariance follows since for $g\in G_1$ we have $\rho_{std}(g)\omega = g\omega g^{-1}$ and $\rho_{spin}(g)\alpha=g\cdot \alpha$.
Note that $\cV_{\rho_{tw}}\subset \cV_{\rho_{std}}$ in a $G_1$-equivariant way.

\autoref{prop:Cl10} implies that the kernel of $\phi$ is $\cV_{\rho_{tw}}\otimes \Cl^{1,0}_+$ and the image is $\Cl^{1,0}_+$.
Recall also that we have a short exact sequence
$$
0\to \Cl^{1,0}_+\to \cV_{\rho_{spin}} \to \Cl^{0,1}_+\to 0
$$
The polarization identifies the $(0,1)$ bundle with the dual of the $(1,0)$ bundle:
$$
\Cl^{0,1}_+ \cong \left( \Cl^{1,0}_+ \right)^\dual
$$
So the map $\phi$ defined above induces an isomorphism
$$
\phi:\cV_{\rho_{tw}}\otimes \left(\Cl^{1,0}_+\right)^\dual \to \Cl_+^{1,0}
$$
Noting that $\dim_\bC \Cl^{1,0}_+=2^n$ and taking determinants gives an isomorphism of line bundles:
\begin{align}
\label{eqn:deg_Hg}
\cV_{\rho_{tw}}^{2^n}\cong \det \left(\Cl^{1,0}_+\right)^2
\end{align}
This calculation was also done by Maulik \cite[Prop. 5.8]{Maulik} (note that the extra power appearing there is due to the non-algebraically closed situation).

\newcommand{\GO}{\operatorname{GO}}

\begin{remark}
 An equivalent description of the period domain is that it parametrizes conjugacy classes of homomorphisms $h:\bS\to G$ for an appropriate group $G$ (see \autoref{def:Hg_str}).
 The action of $G$ is by conjugating the homomorphisms. 
 
 The domain $\Omega_n$ naturally parametrizes homomorphisms $h:\bS\to \GO_{2,n}(\bR)$.
 Taking the standard representation of $\GO_{2,n}(\bR)$ leads to one collection of equivariant bundles over the domain.
 They give a variation of Hodge structure of K3 type.
 
 But one can always lift such $h$ to $\tilde{h}:\bS\to \CSpin$ and take the representation of $\CSpin$ on the Clifford algebra.
 This leads to the other collection of equivariant bundles.
 They give a variation of Hodge structure of weight $1$.
 
 A different way to say it is that this leads to an embedding of Shimura data (see e.g. \cite[Prop. 5.7]{Maulik}).
\end{remark}

A consequence of the above discussion, necessary for computing Lyapunov exponents, is the following.

\begin{theorem}[Deligne {\cite[Prop. 5.7]{Deligne_K3}}, Kuga-Satake \cite{KS} ]
\label{thm:KS}
 Let $H\to B$ be a polarized variation of Hodge structure of K3 type over a base $B$.
 Then there exists a finite unramified cover $B'\to B$ and a weight $1$ polarized variation of Hodge structure $\KS(H)$ over $B'$, satisfying:
 \begin{itemize}
  \item If the local system of $H$ comes from a representation
  $$
  \pi_1(B)\xrightarrow{\rho} \SO_{2,n}\curvearrowright \bR^{2,n}
  $$
  then the local system of $\KS(H)$ comes from a lift of $\rho$ to the spin group:
  $$
  \pi_1(B')\xrightarrow{\rho'} \Spin_{2,n}\curvearrowright \left(\operatorname{Spin Rep}\right)^{2^{\lfloor n/2\rfloor}}
  $$
  Here $\operatorname{Spin Rep}$ denotes the irreducible spin representation if $n$ is odd, or the direct sum of the two irreducible spin representations if $n$ is even.  
  \item The rank of the local system for $\KS(H)$ is $2^{n+1}$, and there is an isomorphism of holomoprhic line bundles
  $$
  \left(H^{2,0}\right)^{2^{n-1}} \cong \det \left( \KS(H)^{1,0} \right)
  $$
  The rank of $\KS(H)^{1,0}$ is $2^n$, so $\det$ means its $2^n$ exterior power.
 \end{itemize}
\end{theorem}

\begin{proof}
 Fix a base point $b\in B$.
 This gives a monodromy representation $\pi_1(B,b) \to \GO (H,I)$, as well as a $\pi_1$ equivariant map from the universal cover of $B$ to the period domain classifying Hodge structures of the same type:
 $$
 \tilde{B}\to \Omega_n
 $$
 Pick a finite index subgroup $\Gamma$ of $\pi_1(B,b)$ such that the monodromy lifts to the spin group, and acts on $\Omega_n$ freely.
 The variation of Hodge structure of weight $1$ from the Kuga-Satake construction pulls back to $\tilde{B}$ and is equivariant for the action of $\Gamma$.
 Taking the quotient gives the desired variation over a finite cover of $B$. 
 
 The claim about the nature of the local system on $\KS(H,I)$ follows from the construction of automorphic bundles.
 The degree of the Hodge bundle is computed in \autoref{eqn:deg_Hg}.
\end{proof}

\begin{remark}
 The weight $1$ variation of Hodge structure obtained above from the Kuga-Satake construction has a large endomorphism ring.
 In fact, the weight $2$ variation $H$ embeds into $\End(\KS(H,I))$.
\end{remark}

\section{Preliminaries from dynamics}
\label{sec:Prelims_Dyn}

This section contains a basic discussion of flows, cocycles, and associated Lyapunov exponents.
The Oseledets theorem, as well as a more geometric point of view on it, is in~\autoref{subsec:Oseledets}.
Computations for spin representations, necessary for using the Kuga-Satake construction, are in~\autoref{subsec:spin_reps}.

\subsection{Oseledets theorem}
\label{subsec:Oseledets}

First we recall the Multiplicative Ergodic Theorem.
For a clear introduction to this topic, see Ledrappier's lectures \cite{Ledrappier}.

\begin{theorem}[Oseldets]
\label{thm:Oseledets}
 Let $(X,\mu)$ denote a probability measure space, equipped with an ergodic flow $g_t$.
 Suppose $V\to X$ is a vector bundle over $X$, equipped with a norm $\norm{-}$ and with a lift of the $g_t$-action.
 This means we have linear maps between fibers:
 $$
 g_t(x):V_x\to V_{g_tx}
 $$
 Assume $L^1$-boundedness of the linear maps, i.e.
 \begin{align}
 \label{eq:Os_int}
 \int_X \sup_{t\in [-1,1]} \norm{ g_t(x) }_{V_x\to V_{g_t x} }d\mu(x) <\infty
 \end{align}
 Then there exist numbers $\lambda_1>\lambda_2>\cdots >\lambda_r$, called Lyapunov exponents, and a measurable $g_t$-invariant decomposition of the bundle
 $$
 V_x = \bigoplus_i V^{\lambda_i}_x
 $$
 such that for $\mu$-a.e. $x\in X$ and $v\in V^{\lambda_j}_x$ we have the asymptotic growth of norm:
 $$
 \lim_{t\to \pm\infty} \frac 1t \log{\norm{g_t v }} = \lambda_j
 $$
 If say $\lambda_j>0$ this means that vectors in $V^{\lambda_j}$ are exponentially expanded/contracted when $t$ goes to plus/minus infinity.
\end{theorem}

\begin{remark}
 When the vector bundle $V$ comes from a variation of Hodge structures over a hyperbolic Riemann surface, and $g_t$ is the geodesic flow, the boundedness assumption is automatic.
 By \autoref{rmk:VHS_bdd}, the integrand in \autoref{eq:Os_int} is uniformly bounded.
\end{remark}

\paragraph{Covers.}
Let $(X',\mu')\to (X,\mu)$ be a finite cover with a compatible lift of the flow $g_t$, which remains ergodic.
Let $V'$ be the pullback of the bundle $V$ to a bundle over $X'$; the action of the flow extends naturally.

The Oseldets theorem also applies to $(X',\mu')$ and $V'$.
The Lyapunov exponents agree with those for $V$, and the measurable decomposition is the pullback of the one on $V$.

\paragraph{Principal bundles.}
A vector bundle $V\to X$ corresponds to a principal bundle $P\to X$.
If the fibers of $V$ are isomorphic to $\bR^n$, the fibers of $P$ are isomorphic to linear automorphisms of $\bR^n$.
Moreover, $P$ carries a free action on the right of a group of linear automorphisms.

The fiber $P_x$ of $P$ over $x$ is defined to be the space of linear isomorphisms from $V_x$ to $\bR^n$.
The natural action of $\GL_n(\bR)$ is on the right, by post-composing with self-maps of $\bR^n$.

Often the vector bundle carries extra structure, such as an invariant bilinear form.
In this case, the fibers of the principal bundle are defined as isomorphisms preserving the extra structure.
The action on the right is by the group of linear isomorphisms of $\bR^n$ preserving the extra structure.

\paragraph{Changing the representation.}
Consider a principal bundle $P\to X$ with structure group $G$ acting on $P$ on the right.
Given a representation $\rho:G\to \GL(\bR^m)$ we can form an associated vector bundle over $X$ via $W:=P\times_G \bR^m \to X$, where
$$
W:=P\times_G \bR^m :=\{ (p,w)\sim (pg,\rho(g^{-1})w ) \vert \forall g\in G\}
$$
Suppose the vector bundle $V$ which produced the principal bundle $P$ carried a flow $g_t$.
Then so does the principal bundle $P$, and so will the new vector bundle $W$.
The relation between the Lyapunov exponents of $V$ and $W$ is explained in the next section.

\begin{remark}
 Typically the above abstract construction is unnecessary.
 The vector bundle $V$ corresponds to the standard representation of $G$ and one can apply tensor operations to $V$ and obtain new bundles, which will contain most representations.
 
 In the case relevant to the Kuga-Satake construction, this direct approach does not work.
 The representation that appears is a spin representation, and it does not occur in tensor powers of the standard representation.
 This leads to the abstract considerations above, and in the next section.
\end{remark}

Describing Lyapunov exponents in this general setting depends on some Lie theory, described below.
More details are in Bump's monograph \cite[Part III]{Bump}.

\newcommand{\res}{\text{res}}

\paragraph{Structure of semisimple Lie groups.}
Let $G$ be a connected semisimple Lie group.
It has an Iwasawa decomposition
$$
G=KAK
$$
where $K$ is a maximal compact subgroup, and $A$ is a maximal $\bR$-split torus.
Denote by $\Phi$ the root system of the complexification $G_\bC$ and let $\Phi_\res$ be the restricted root system associated to the real form $G$.
They are related by a map $r:\Phi\to \Phi_\res$.
Recall also that $\Phi_\res$ is contained in $\fraka^\dual$, the dual of the Lie algebra of $A$.

Given a representation $\rho$ of $G$, its weights $\Sigma_\rho$ are contained in $\Phi$.
We also have the restricted weights $r(\Sigma_\rho)$, the projection of the weights to $\Phi_\res$.
All weights are taken with multiplicities.

Given a diagonalizable element $a\in G$, after conjugation assume it is in $A$.
To describe the eigenvalues of $\rho(a)$ on $V_\rho$ consider the restricted weights 
$$
r(\Sigma_\rho)\subset \Phi_\res\subset \fraka^\dual
$$
Then the logs of the eigenvalues of $\rho(a)$ equal the evaluation of the restricted weights against $\log a\in \fraka = \Lie A$.

\paragraph{Universal Lyapunov exponents}
The following reformulation of the Oseledets theorem is due to Kaimanovich \cite{Kaimanovich}.

Let $g_t\curvearrowright (X,\mu)$ be a an ergodic flow on a probability space.
Assume the flow lifts to a principal $G$-bundle $P\to X$, where $G$ is a semisimple Lie group and the lift satisfies an appropriate $L^1$-boundedness assumption.

Let $\Phi, \Phi_\res$ be the root system and its restricted counterpart.
Let also $\fraka$ denote the Lie algebra of a maximal split torus of $G$, and $\fraka^\dual$ its dual.
Denote by $\fraka_+$ the positive Weyl chamber of $\fraka$.

Then there exists a \textbf{Lyapunov vector} $\Lambda\in \fraka_+$ which controls Lyapunov exponents as follows.
Given a representation $\rho$ of $G\curvearrowright V_\rho$ with weights $\Sigma_\rho\subset \Phi$, form the associated vector bundle $P\times_G V_\rho$ .
Then its Lyapunov exponents are given by evaluating the restricted weights $r(\Sigma_\rho)$ on the Lyapunov vector $\Lambda$.

\begin{example}
 Suppose $G=\SL_{n+1}(\bR)$.
 Then the maximal split torus is
 $$
 A=\left\lbrace\diag(e^{x_1},\ldots,e^{x_{n+1}})\middle\vert \sum x_i=0 \right\rbrace
 $$
 The two root systems $\Phi$ and $\Phi_\res$ agree and are of type $A_n$.
 The Lie algebra $\fraka$ is given by
 $$
 \fraka = \left\lbrace(x_1,\ldots,x_{n+1})\in \bR^{n+1}\middle\vert \sum x_i =0 \right\rbrace
 $$
 The positive Weyl chamber is given by
 $$
 \fraka_+ = \left\lbrace (x_1,\ldots,x_{n+1})\in \bR^{n+1}\middle\vert \sum x_i =0, x_1\geq \cdots \geq x_{n+1} \right\rbrace
 $$
 The dual of $\fraka$ is given as a quotient
 $$
 \fraka^\dual = \left. \left\lbrace(\xi_1,\ldots,\xi_{n+1})\in \bR^{n+1} \right\rbrace \middle/\sum \xi_i = 0\right.
 $$
 The weights of the standard representation of $\SL_{n+1}(\bR)$ are the standard basis vectors $e_i\in \bR^{n+1}$ projected to $\fraka^\dual$.
\end{example}

\begin{remark}
 From the description of the Lyapunov spectrum as a vector in a Weyl chamber, degeneracy in the spectrum has a geometric interpretation.
 For example, coincidence of two exponents (in the standard representation of $\SL_{n+1}$) is the same as the Lyapunov vector hitting a wall of the positive chamber $\fraka_+$.
\end{remark}
The next section contains a detailed analysis of an example involving spin representations.

\subsection{Computations with Spin representations}
\label{subsec:spin_reps}
The analysis of the root systems arising in the case of $\so_{2,n}(\bR)$ is divided in two cases, depending on the parity of $n$.
\begin{center}
\textbf{Type $B, \so_{2,2k-1}(\bR)$}
\end{center}
In this case $n=2k-1$ is odd.
The Lie algebra $\so_{2k+1}(\bC)$ has root system of type $B_k$.
In $\bR^k$ with standard basis $\{e_i\}$ the roots and fundamental weights are

\begin{tabular}{lll}
  Simple roots & $\alpha_i = e_i-e_{i+1}$ 	& for $ i=1\ldots k-1$\\
	       & $\alpha_k = e_k$ 		& \\
  Fund. weights& $\varpi_i=e_1+\cdots+e_i$ 	& for $i=1\ldots k-1$\\
	      & $\varpi_k = \frac 12 (e_1+\cdots +e_k)$ 	&
\end{tabular} 

\noindent The restricted root system of $\so_{2,n}(\bR)$ is of type $B_2$, with simple roots 
\begin{align*}
 \beta_1&:=f_1-f_2\\
 \beta_2&:=f_2
\end{align*}
The map (see \cite[Table 4]{Vinberg}) to the restricted root system is
$$
r(\alpha_i)=\beta_i \text{ for } i=1,2 \text{ and } r(\alpha_i)=0 \text{ otherwise}
$$
Solving the equations for $e_i$ we find
\begin{align*}
r(e_1)&=\beta_1+\beta_2 = f_1\\
r(e_2)&=\beta_2 = f_2\\
r(e_j)&=0 	\text{ otherwise}
\end{align*}
The weights of the spin representation (see \cite[Thm. 31.2]{Bump}) occur with multiplicity one and are
$$
\left\lbrace\frac 12 \left(\pm e_1\pm e_2 \cdots\pm e_k\right)\right\rbrace
$$
Mapping them to the restricted root system, we find that each of the following occurs with multiplicity $2^{k-2}=2^{(n-3)/2}$
$$
\left\lbrace\frac 12 (f_1+f_2),\frac 12 (f_1-f_2), -\frac 12 (f_1-f_2), -\frac 12 (f_1+f_2)\right\rbrace
$$
For the standard representation of $\so_{2,2k-1}$ the weights in $B_k$ are
$$
\{e_1,\ldots, e_k, 0, -e_k,\ldots, -e_1\}
$$
Therefore, the restricted weights are (where $0$ occurs with multiplicity $2k-3=n-2$)
$$
\{f_1,f_2,0,\ldots,0,-f_2,-f_1 \}
$$
Applying the discussion concerning Lyapunov exponents corresponding to different representations, the next result follows.

\begin{proposition}
\label{prop:odd_spin}
 Suppose $g_t\curvearrowright(X,\mu)$ is a probability space equipped with an ergodic flow.
 Let $V\to X$ be a vector bundle equipped with a metric of signature $(2,2k-1)$ and a lift of the action of $g_t$ preserving the metric.
 Suppose $V$ also carries a positive-definite metric, not invariant under $g_t$, but for which the flow is integrable in the sense of the Oseledets theorem.
 
 Let $W\to X'$ be the vector bundle derived from $V$, after perhaps passing to a finite cover, where the fibers correspond to the spin representation.
 The structure group changes from $O_{2,2k-1}$ to its spin double cover.
 
 Then the Lyapunov exponents of $(g_t,V)$ are
 $$
 \lambda_1 \geq \lambda_2 \geq 0 \geq \cdots \geq 0 \geq -\lambda_2 \geq -\lambda_1
 $$
 where $0$ occurs with multiplicity $2k-3$.
 
 The Lyapunov exponents of $(g_t,W)$ are
 $$ 
 \frac 12 (\lambda_1+\lambda_2)\geq \frac 12 (\lambda_1 -\lambda_2) \geq -\frac 12 (\lambda_1-\lambda_2) \geq -\frac 12 (\lambda_1+\lambda_2)
 $$
 Each of them occurs with multiplicity $2^{k-2}$.
 Note the dimension of the spin representation is $2^k$.
\end{proposition}

\begin{center}
\textbf{Type $D, \so_{2,2k-2}(\bR)$}
\end{center}
In this case $n=2k-2$ is even.
The Lie algebra $\so_{2k}(\bC)$ has root system of type $D_k$.
In $\bR^k$ with standard basis $\{e_i\}$ the roots and fundamental weights are

\begin{tabular}{lll}
  Simple roots & $\alpha_i = e_i-e_{i+1}$ 	& for $ i=1\ldots k-1$\\
	       & $\alpha_k = e_{k-1}+e_n$ 	& \\
  Fund. weights& $\varpi_i=e_1+\cdots+e_i$ 	& for $i=1\ldots k-2$\\
		&$\varpi_{k-1} = \frac 12 (e_1+\cdots+e_{n-1} -e_k)$ & \\
		&$\varpi_k = \frac 12 (e_1+\cdots+e_{k-1} +e_k)$ &
 \end{tabular}

\noindent The restricted root system of $\so_{2,n}(\bR)$ is of type $B_2$, with simple roots 
\begin{align*}
 \beta_1&:=f_1-f_2\\
 \beta_2&:=f_2
\end{align*}
The map (see \cite[Table 4]{Vinberg}) to the restricted root system is for $k\geq 4$ 
$$
r(\alpha_i)=\beta_i \text{ for } i=1,2 \text{ and } r(\alpha_i)=0 \text{ otherwise}
$$
For $k=3$, the map is
$$
r(\alpha_1)=\beta_1\hskip 0.2in r(\alpha_2)=r(\alpha_3)=\beta_2
$$
Solving the equations for $e_i$ we find independently of $k$ that
\begin{align*}
r(e_1)&=\beta_1+\beta_2 = f_1\\
r(e_2)&=\beta_2 = f_2\\
r(e_j)&=0 	\text{ otherwise}
\end{align*}

There are two spin representation $V(\varpi_{k-1})$ and $V(\varpi_k)$.
Their weights (see \cite[Thm. 31.2]{Bump}), each with multiplicity one, are
$$
\left\lbrace\frac 12 \left(\pm e_1\pm e_2 \cdots\pm e_k\right)\right\rbrace
$$
The representation $V(\varpi_{k-1})$ contains all summands with an odd number of minus signs, while $V(\varpi_k)$ those with an even number.

Mapping the weights to the restricted root system, we find that each representation has the following weights, each with multiplicity $2^{k-3}=2^{(n-4)/2}$
$$
\left\lbrace\frac 12 (f_1+f_2),\frac 12 (f_1-f_2), -\frac 12 (f_1-f_2), -\frac 12 (f_1+f_2)\right\rbrace
$$
For the standard representation of $\so_{2,2k-2}$ the weights in $B_k$ are
$$
\{e_1,\ldots, e_k, -e_k,\ldots, -e_1\}
$$
Therefore, the restricted weights are (where $0$ occurs with multiplicity $2k-4=n-2$)
$$
\{f_1,f_2,0,\ldots,0,-f_2,-f_1 \}
$$
Applying the discussion concerning Lyapunov exponents corresponding to different representations, the next result follows.

\begin{proposition}
\label{prop:even_spin}
 Suppose $g_t\curvearrowright(X,\mu)$ is a probability space equipped with an ergodic flow.
 Let $V\to X$ be a vector bundle equipped with a metric of signature $(2,2k-2)$ and a lift of the action of $g_t$ preserving the metric.
 Suppose $V$ also carries a positive-definite metric, not invariant under $g_t$, but for which the flow is integrable in the sense of the Oseledets theorem.
 
 Let $W\to X'$ be the vector bundle derived from $V$, after perhaps passing to a finite cover, where the fibers correspond to the direct sum of the two spin representations.
 The structure group changes from $O_{2,2k-2}$ to its spin double cover.
 
 Then the Lyapunov exponents of $(g_t,V)$ are
 $$
 \lambda_1 \geq \lambda_2 \geq 0 \geq \cdots \geq 0 \geq -\lambda_2 \geq -\lambda_1
 $$
 where $0$ occurs with multiplicity $2k-4$.
 
 The Lyapunov exponents of $(g_t,W)$ are
 $$ 
 \frac 12 (\lambda_1+\lambda_2)\geq \frac 12 (\lambda_1 -\lambda_2) \geq -\frac 12 (\lambda_1-\lambda_2) \geq -\frac 12 (\lambda_1+\lambda_2)
 $$
 Each of them occurs with multiplicity $2^{k-2}$.
 Note the dimension of the sum of the two spin representations is $2^k$.
\end{proposition}

\section{Top Lyapunov exponent for K3s}
\label{sec:top_exp}

\paragraph{Setup.}
Throughout this section, $C$ is a fixed compact Riemann surface with finitely many punctures $S\subset C$.
Let $H$ be a polarized variation of Hodge structure of K3 type over $C\setminus S$, which is not locally isotrivial.
Recall (\autoref{def:VHS}) this gives a local system $H_\bR$ with a flat symmetric bilinear form $I(-,-)$ of signature $(2+,n-)$, as well as a decomposition of the complexification:
$$
H_\bC = H^{2,0}\oplus H^{1,1}\oplus H^{0,2}
$$ 
Assume that $C\setminus S$ is of hyperbolic type, i.e. carries a complete finite volume hyperbolic metric.
Let $g_t$ denotes the unit speed geodesic flow for this metric, defined on the unit tangent bundle $T^1(C\setminus S)$.
By \autoref{rmk:VHS_bdd} the integrability condition of the Oseledets theorem is automatically satisfied.

The goal of this section is to give two different proofs of the next result.

\begin{theorem}
\label{thm:top_exp}
 With the setup as above, consider the Lyapunov exponents of the cocycle for $g_t$ induced by the local system:
 $$
 \lambda_1\geq \lambda_2 \geq 0 \geq \cdots \geq 0 \geq -\lambda_2\geq -\lambda_1
 $$
 The multiplicity of zero in the spectrum is $n-2$, where the signature of the flat metric is $(2+,n-)$.
 
 Then we have that $\lambda_1>\lambda_2$ and moreover
 $$
 \lambda_1 = \frac 12 \cdot \frac{\deg H^{2,0}}{\deg K_C(\log S) }
 $$
 Here $K_C(\log S)$ denotes the cotangent bundle of $C$ with logarithmic singularities along $S$.
 The line bundle $H^{2,0}$ admits an algebraic extension to all of $C$ and its degree is taken for that extension.
\end{theorem}

\begin{remark}
 There are two degenerate cases of the above theorem, when $n=1$ and $n=2$.
 
 When $n=1$, the spectrum is $\lambda_1>0>-\lambda_1$ and the formula holds.
 
 When $n=2$, the spectrum is $\lambda_1>\lambda_2>-\lambda_2>-\lambda_1$.
 In this case the formula also holds, but $\lambda_2$ can also be computed.
 Indeed, the variation of Hodge structure of K3 type will be the tensor product of two variations of weight $1$.
 This corresponds to the exceptional isomorphism $\so_{2,2}\cong \fraksl_2 \oplus \fraksl_2$.
\end{remark}

\subsection{Top exponent via Kuga-Satake}

Invoking the results obtained in previous sections, the computation in \autoref{thm:top_exp} reduces to the case of weight $1$.

By \autoref{thm:KS}, after passing to a finite cover there is a variation of Hodge structure of weight $1$, denoted $\KS(H)$.
Passing to finite unramified covers does not affect Lyapunov exponents or ratios of degrees of line bundles.

The rank of the local system $\KS(H)$ is $2^{n+1}$ and by \autoref{prop:odd_spin} and \autoref{prop:even_spin}, the exponents which occur in it are
$$ 
\frac 12 (\lambda_1+\lambda_2)\geq \frac 12 (\lambda_1 -\lambda_2) \geq -\frac 12 (\lambda_1-\lambda_2) \geq -\frac 12 (\lambda_1+\lambda_2)
$$
Each occurs with multiplicity $2^{n-1}$.
Applying the formula for the sum of positive exponents from \cite[Eqn. 3.11]{EKZ} it follows that
$$
2^{n-1}\cdot \lambda_1 = 2^{n-1} \cdot \left(\frac 12 (\lambda_1+\lambda_2)+ \frac 12 (\lambda_1 -\lambda_2) \right ) = \frac 12 \cdot \frac{\deg \KS(H)^{1,0}}{\deg K_C(\log S) }
$$
However, by \autoref{thm:KS} the degrees of $H^{2,0}$ and $\KS(H)^{2,0}$ are related by
$$
\deg \left(H^{2,0}\right) = 2^{n-1}\deg \left(\KS(H)^{2,0}\right)
$$
This implies the claimed formula.
The spectral gap inequality is discussed in the next section.
\hfill \qed

\begin{remark}
 In \cite{EKZ}, the authors normalize the hyperbolic metric to have curvature $-4$.
 Their formula thus has a factor of $2$ in front of the ratio of degrees, whereas ours has a factor of $\frac 12$.
 See \cite[3.11]{EKZ} and the discussion following it for a description of their normalization.
\end{remark}

\subsection{Top exponent via integration by parts}

This section contains an alternative approach to the formula for the top exponent.
Just like in the weight $1$ case, one can take a typical (multivalued) flat section and compute its growth by using the curvature of bundle.
There are several intermediate steps which reduce the calculation to this argument.

 \noindent\textbf{Step 1: Spectral gap.}
 There are two methods to establish the inequality $\lambda_1>\lambda_2$.
 One is to invoke the result of Eskin-Matheus \cite[Theorem 1]{Eskin_Matheus}.
 It gives a criterion for simplicity which applies directly in this case. 
 Note this also gives that $\lambda_2>0$.
 
 Another method is to note that in the Kuga-Satake construction, the second exponent is $\lambda_1-\lambda_2$.
 Now, the arguments from \cite{zero_exp} apply to the Kuga-Satake family and show that one cannot have a zero exponent in the weight $1$ case.
 
 \noindent{\textbf{Step 2: Reducing to a single vector.}} 
 Denote by 
 $$
 g^{\theta}_t:T^1(C\setminus S) \to T^1(C\setminus S)
 $$
 the geodesic flow on the unit tangent  bundle of $C\setminus S$.
 The parameter $\theta$ indicates conjugation by a rotation of angle $\theta\in [0,2\pi)$.
 We thus have a measurable Oseledets decomposition which is $g_t$-invariant:
 $$
 E^{\lambda_1}\oplus E^{[\lambda_2,-\lambda_2]}\oplus E^{-\lambda_1}
 $$
 The middle term has a further refinement, which we don't need.
 It is essential, however, that $\lambda_1>\lambda_2$.
 
 Consider the dynamics on the bundle of Grassmanians of isotropic lines in the local system $H$.
 Denote it by $\Gr_I(1,H)$ and its fiber at a point $x$ by $\Gr_I(1,H_x)$.
 From the spectral gap, for all $x$, for a.e. $\theta$ and Lebesgue-a.e. isotropic vector $\phi_x$ such that $[\phi_x]\in \Gr_I(1,H_x)$ we have
 $$
 \lim_{T\to \infty} \frac 1T \log \norm{g_T^\theta \phi_x} = \lambda_1
 $$
 The above formula fails only when $\phi_x$ is contained in $E^{[\lambda_2,-\lambda_2]}\oplus E^{-\lambda_1}$ and this is a proper (algebraic) subset of $\Gr_I(1,H_x)$.
 
 The key claim is that \emph{for any} $x\in C\setminus S$ and \emph{any isotropic} $\phi_x\in H_x$, we have
 \begin{align}
 \label{eqn:Lyap=avg}
 \lambda_1 = \lim_{T\to \infty} \frac 1 {T} \int_{0}^{2\pi}  \log \norm{g_T^\theta \phi_x} \frac  {d\theta }{2\pi}
 \end{align}
 For this, let $\eta$ be some Lebesgue-class probability measure on $\Gr_I(1,H_x)$.
 Then an $\eta$-typical space is Lyapunov-regular, i.e. the Oseledets theorem holds.
 So for a fixed $\theta\in[0,2\pi)$ we have
 $$
 \lambda_1 = \lim_{T\to \infty} \frac 1 T \int \log \norm{g_T^\theta v} \,d\eta(v)
 $$
 We can now average in $\theta$ and exchange the order of integration and limits:
 \begin{align*}
  \lambda_1 &= \lim_{T\to \infty}  \int_0^{2\pi} \left( \frac 1 T \int \log \norm{g_T^\theta v}\,d\eta(v)\right) \frac{d\theta}{2\pi}\\
  & = \int \left( \lim_{T\to \infty} \frac 1 T \int_0^{2\pi} \log \norm{g_T^{\theta} v}  \frac{d\theta}{2\pi} \right) d\eta(v)
 \end{align*}
 The exchanges are allowed because in this case, we have a uniform bound on the growth of norms, so $\frac 1 T({\log \norm{g_T v}})$ is uniformly bounded.
 
 In the next step, we show that the inner integral above, which equals the one in \autoref{eqn:Lyap=avg}, is independent of the choice of isotropic vector.
 So the extra averaging in $\eta$ is unnecessary, whence the claim.
 
 \noindent{\textbf{Step 3: Integration by parts.}}
 To proceed, take the universal cover of $C\setminus S$, which is the hyperbolic unit disk:
 $$
 \bD \to C \setminus S
 $$
 Fix some $x\in \bD$ and choose $\phi_x\in H_x$, a real isotropic vector in the fiber of $H$ over $x$.
 Extend using the Gauss-Manin connection $\phi_x$ over $\bD$ to a global flat section denoted $\phi$.
 The section is pointwise isotropic on $\bD$, since $\phi_x$ is isotropic, and the flat connection preserves the indefinite form.
 
 We have the decomposition of $\phi$ into types:
 $$
 \phi=\phi^{2,0}\oplus \phi^{1,1}\oplus \phi^{0,2}
 $$
 The isotropy condition gives
 $$
 \norm{\phi^{2,0}}^2 + \norm{\phi^{0,2} }^2 = \norm{\phi^{1,1}}^2
 $$
 Using further that $\phi$ is real, we find $\norm{\phi}^2 = 4\norm{\phi^{0,2}}^2$.
 The norm $\norm{-}$ is for the positive-definite inner product.
 
 We now rewrite \autoref{eqn:Lyap=avg} as
 $$
 \lambda_1 = \lim_{T\to \infty} \frac 1 {2T} \int_0^{2\pi}  \log\left( \norm{g_T^\theta \phi_x}^2 \right) \frac{d\theta}{2\pi}
 $$
 Next we use integration by parts:
 $$
 \log \left( \norm{g_T^\theta \phi_x}^2 \right) = \int_0^T \frac{d}{dt} \log \left( \norm{g_t^\theta \phi_x}^2 \right) dt - \log \left( \norm{\phi_x}^2 \right)
 $$
 The term $\log \norm{\phi_x}^2$ goes away after dividing by $T$ and letting $T\to \infty$.
 For fixed $t\in[0,T]$ apply Green's theorem on the hyperbolic disc $\bD_t\subset \bD$ of (hyperbolic) radius $t$: 
 $$
 l_t \int_0^{2\pi} \frac{d}{dt} \log\left(\norm{g_t^\theta\phi_x }^2\right) \frac{d\theta}{2\pi} 
 = 
 \int_{\bD_t} \left[\Laplace_\hyp \log\left( \norm{\phi}^2 \right) \right] d\Vol_{\hyp}
 $$
 Above, $l_t$ denotes the hyperbolic length of the boundary of the disc $\bD_t$.
 The hyperbolic laplacian is denoted $\Laplace_\hyp$ and the hyperbolic volume form by $d\Vol_\hyp$.
 
 For any function $f$ we have
 $$
 \left(\Laplace_\hyp f\right) d\Vol_\hyp = \sqrt{-1}\del\delbar f
 $$
 For a holomoprhic section $\phi$ of a line bundle $\cL$, we have the following relation to the curvature $\Omega$ of $\cL$:
 $$
 \del \delbar \log\left( \norm{\phi}^2 \right) = -\frac{ \ip{\Omega\phi}{\phi} }{\norm{\phi}^2}
 $$
 Indeed, this can be found, for example, in \cite[Lemma 3.1]{sfilip_ssimple}; the second term vanishes since $\phi$ is a section of a line bundle.
 Note that the right hand side above is independent of $\phi$.
 
 Combining all of the facts above, it follows that
 $$
 \lambda_1 = \lim_{T\to \infty} \frac{1}{2T}\int_{0}^T \left( \int_{\bD_t} -\sqrt{-1}\frac{ \ip{\Omega \phi^{0,2}}{\phi^{0,2}}} {\norm{\phi^{0,2}}^2 }\right) \frac{dt}{l_t}
 $$
 Note that $\norm{\phi}^2$ was replaced by $4\norm{\phi^{0,2}}^2$, but after taking logs in the limit $T\to \infty$ the factor of $4$ goes away.
 
 The formula above used a base-point $x\in \bD$, which can be averaged out.
 To this end, define
 $$
 A_T(x) = \frac 1{2T} \int_0^T \left( \int_{\bD_t(x)} \Phi(z)d\Vol_\hyp(z) \right)\frac {dt}{l_t} 
 $$
 Here $\bD_t(x)$ is the disc of hyperbolic radius $t$ centered at $x$, and $\Phi$ is defined by
 $$
 \Phi(z)d\Vol_\hyp(z) = -\sqrt{-1} \frac{ \ip{\Omega\phi^{0,2}}{\phi^{0,2}} }{\norm{\phi^{0,2}} }
 $$
 Note that $\Phi$ is in fact independent of the choice of holomorphic section $\phi$, and descends to $C\setminus S$.
 
 Next, average the quantity $A_T(x)$ for $x$ in a fundamental domain of the universal cover $\bD\to C\setminus S$.
 The interior integral in the definition of $A_T(x)$, when averaged, becomes:
 $$
 \Vol_\hyp(\bD_t) \int_{C\setminus S}\Phi(z)d\Vol_\hyp(z)
 $$
 We therefore find that
 \begin{align}
 \label{eqn:hyp_limit}
 \lambda_1 = \frac{\int_{C\setminus S} \Phi(z)d\Vol_\hyp(z) } {\int_{C\setminus S}1d\Vol_\hyp} \cdot
 \lim_{T\to \infty} \frac{1}{2T} \int_0^T \frac {\Vol_\hyp(\bD_t)} {l_t} dt
 \end{align}
 To compute the limit above, recall in the unit disc model the metric is $ds^2=\frac{4|dz|^2}{(1-|z|^2)^2}$.
 The disc of hyperbolic radius $t$ corresponds to the Euclidean disc of radius $r$ with
 $$
 \log\left(\frac {1+r}{1-r}\right)=t
 $$
 Computing the appropriate integrals, we find
 \begin{align*}
  l_t &= \frac{4\pi r}{1-r^2}\\
  \Vol_\hyp\left( \bD_t \right)&= \frac{4\pi r^2}{1-r^2}
 \end{align*} 
 Thus the ratio of the hyperbolic volume of the disc to the hyperbolic length of its boundary tends to $1$ as the hyperbolic radius approaches infinity.
 Thus in \autoref{eqn:hyp_limit} the factor corresponding to the limit is $\frac 12$. 
 
 The ratio of integrals appearing in \autoref{eqn:hyp_limit} is also equal to the ratio of degrees of line bundles.
 Indeed, for a line bundle $\cL$ we have the Chern class in terms of curvature as
 $$
 c_1(\cL) = \frac{\sqrt{-1}}{2\pi} \left[\Omega_\cL \right]
 $$
 Following Mumford \cite{Mumford}, the degree of a line bundle can be computed using a singular hermitian metric.
 For the log-canonical, the degree can be computed using the complete hyperbolic metric. 
 The bundle $H^{2,0}$ has a canonical algebraic extension across the punctures, and the Hodge metric provides again a good metric in the sense of Mumford.
 
 To account for the negative sign in the definition of $\Phi(z)$ and the curvature of $H^{0,2}$, note that $\deg H^{0,2}=-\deg H^{2,0}$.
 This leads to the formula
 $$
 \lambda_1 = \frac 12 \cdot \frac{\deg H^{2,0}}{\deg K_C(\log S) }
 $$
\hfill \qed

\section{Examples}
\label{sec:examples}

This section contains some explicit examples of families of K3 surfaces.
The examples are all projective and come with a natural (quasi)-polarization.
They are determined by the degree of the polarization, and are as follows:
\begin{description}
 \item[degree $2$] A double cover of $\bP^2$ ramified over a sextic curve.
 \item[degree $4$] A quartic surface in $\bP^3$.
 \item[degree $6$] The intersection of a quadric and cubic hypersurfaces in $\bP^4$.
 \item[degree $8$] The intersection of three quadric hypersurfaces in $\bP^5$.
\end{description}
\begin{remark}
 Each of the above examples gives a family with $19$ parameters.
 We discuss in detail the example of quartic surfaces, for the others only providing the numerical data, which is taken from the work of Maulik-Pandharipande \cite{Maulik_Pandharipande}.
\end{remark}

\paragraph{Genericity.}
 Below, the arguments apply to generic polynomials.
 However, unlike the situation in ergodic theory where ``generic" often means no concrete example is possible, the opposite is true.
 Writing down polynomials with integer coefficients by hand, a computer algebra system can certify that they are ``generic" in the sense used below. 

\paragraph{Topological Invariance.}
It is also interesting to note that a perturbation of the parameters leads to the same top Lyapunov exponent.
In the examples below, the family will be over a punctured sphere.
While the monodromy matrices do not change under small perturbations, the position of the punctures will.
This changes the hyperbolic structure and hence the geodesic flow, but not the top Lyapunov exponent.

\paragraph{Monodromy.}
All examples below are given as Lefschetz pencils of complete intersections.
This implies that their monodromy has full Zariski closure in the orthogonal group.
The survey of Peters and Steenbrink \cite[Sect. 2]{Peters_Steenbrink} provides a discussion.

\subsection{Quartics in \texorpdfstring{$\mathbb{P}^3$}{P3} }

Fix two generic homogeneous polynomials $F,G$ of degree $4$ in the variables $X_0,\ldots,X_3$.
One can build a family of K3 surfaces over $\bP^1$ by defining for $[\lambda:\mu]\in \bP^1$ the surface
$$
X_{[\lambda:\mu]} = \{\lambda F + \mu G = 0\}\subset \bP^3
$$
The number of points where the fiber is singular will follow from an analysis of the space of all quartic surfaces.

\paragraph{The universal family of hypersurfaces.}
Homogeneous polynomials of degree $d$ on $\bP^n$ are parametrized by $\bP^{N}$ where $N=\tbinom{n+d}{d}-1$.
Define
$$
\Sigma:= \left\lbrace(x,F) \in \bP^n\times \bP^N \middle\vert F(x)=0 \right\rbrace
$$
It carries the projections $\pi_1:\Sigma\to \bP^n$ and $\pi_2:\Sigma \to \bP^N$.
Denote the co\-ho\-mology rings of the projective spaces by
\begin{align*}
 H^\bullet(\bP^n)& = \bZ[a]/\{a^{n+1}=0\}\\
 H^\bullet(\bP^N)& = \bZ[b]/\{b^{N+1}=0\}
\end{align*}
For a variety $X$, denote its cohomology class by $[X]$ (it is obtained as the \Poincare dual of the fundamental cycle).
Then $[\Sigma]=d\cdot (a\otimes 1) + 1\otimes b$.

To determine the polynomials which give singular hypersurfaces, define
$$
\Sigma_i =\left\lbrace (x,F)\in \bP^n\times \bP^N \middle\vert \partial_{x_i} F(x)=0 \right\rbrace
$$
Then $[\Sigma_i]=(d-1)a\otimes 1  + 1\otimes b$.
The singular locus is the intersection of the $\Sigma_i$:
$$
\Sigma_0 \cap \cdots \cap \Sigma_n = \left\lbrace (x,F) \vert \partial_{x_i}F(X)=0 \hskip 0.1in \forall i=0\ldots n\right\rbrace
$$
Its cohomology class is then
$$
[\Sigma_0\cap\cdots\cap \Sigma_n] = \left( (d-1)a\otimes 1 + 1\otimes b \right)^{n+1}
$$
We are interested in the discriminant locus $\cD\subset \bP^N$ given by
$$
\cD:= (\pi_2)_*\left( \Sigma_0\cap\cdots\cap \Sigma_n  \right)
$$
To determine $[\cD]$, take the coefficient of $a^n\otimes 1$ in the class of the intersection of the $\Sigma_\bullet$.
This yields
$$
[\cD] = (n+1) (d-1)^n b\in H^\bullet (\bP^N)
$$
So $\cD$ is a hypersurface of degree $(n+1)(d-1)^n$.
The case relevant to degree $4$ (i.e. quartic) surfaces in $\bP^3$ has degree of the discriminant locus equal to $108$.

\paragraph{Lefschetz pencils.}
Take two generic homogeneous polynomials $F,G$ of degree $4$ in the variables $X_0,\cdots, X_3$.
Define the family
$$
X=\left\lbrace ([\lambda:\mu],x)\in \bP^1\times \bP^3\vert \lambda F(x) + \mu G(x)=0 \right\rbrace
$$
By the above degree calculation, the fibration $X\to \bP^1$ will have $108$ singular fibers, with nodal singularities.

\paragraph{The Hodge bundle.}
We now compute the bundle over $\bP^1$ whose fiber over $t\in \bP^1$ is $H^0(K_{X_t})$.
Here $X_t$ denotes the fiber of $X\to \bP^1$ over $t$, and $K_{X_t}$ is its canonical bundle.
Note that in order to define it over the singular fibers, one can work instead with $H^2(\cO_{X_t})$.

The variety $X$ is equipped with two natural maps.
The projection $\alpha:X\to \bP^1$ exhibits it as a fibration, while $\beta :X\to \bP^3$ presents $X$ as the blow-up of $\bP^3$ along the curve $\{F=0\}\cap \{G=0\}$.
Let $E\subset X$ be the exceptional divisor of the blow-up, so that
$$
K_X =\beta^* K_{\bP^3} + E
$$
Letting $H$ denote the class of a plane in $\bP^3$ or its pullback to $X$, we thus have
$$
K_X = -4H + E
$$
The class of a fiber $X_t$ is $4H-E$, since this is the strict transform of a quartic in the pencil.
Next, $K_{\bP^1}=-2[pt.]$, therefore 
$$
\alpha^* K_{\bP^1}=-2 (4H-E)
$$
So the relative canonical sheaf is given by
$$
K_{X/\bP^1} = K_X - \alpha^* K_{\bP^1} = 4H-E
$$
In particular $K_{X/\bP^1} = \alpha^*(\cO_{\bP^1}(1))$ so
$$
H^{2,0} = \alpha_*(K_{X/\bP^1}) = \cO_{\bP^1}(1)
$$

\paragraph{Top Lyapunov exponent.}
To summarize, we have a family $X\to \bP^1$ which is smooth away from $108$ points, denoted $S$.
The Hodge bundle was calculated as $H^{2,0}=\cO_{\bP^1}(1)$, and so
\begin{align*}
 \deg K_{\bP^1}(\log S) &= 108-2 = 106\\
 \deg H^{2,0} &= 1\\
 \lambda_1 &= \frac 1 {212}
\end{align*}\

\begin{remark}
 There is a variation of the above construction, for which we refer to \cite[Sect. 5.1]{Maulik_Pandharipande}.
 Take $C_{53}\to \bP^1$ to be a degree $2$ hyperelliptic cover ramified over the $108$ points corresponding to nodal K3s (so, of genus $53$).
 The pulled-back family $\tilde{X}\to C_{53}$ admits a minimal desingularization, so that all the fibers are smooth K3 surfaces.
 We then have $\deg H^{2,0} = 2$ and $\deg K_{C_{53}}=104$, so the top exponent is $\lambda_1 = \frac 1 {104}$.
\end{remark}

\subsection{Other classical examples}

The calculations for the next examples are similar to the one above, although a bit more involved (so, omitted).
The numerical values are taken from \cite[Sect. 6]{Maulik_Pandharipande}.

\paragraph{Intersection of a quadric and cubic in $\bP^4$.}
Pick generic homogeneous polynomials $F$ of degree $2$ and $G$ of degree $3$ in the variables $X_0,\ldots,X_4$.
The intersection of their zero loci gives a smooth $K3$ surface.
Two possible kinds of Lefschetz pencils are obtained by letting $F$ vary and keeping $G$ fixed, or vice-versa.
\begin{description}
 \item[Case 1] Pick generic $F_1,F_2$ of degree $2$ and let
 $$
 X_{[\lambda:\mu]} = \{ \lambda F_1 +\mu F_2 = 0 \} \cap \{G=0\}
 $$
 This determines a fibration over $\bP^1$ with $7$ nodal fibers and $\deg H^{2,0}=1$, therefore $\lambda_1 = \frac 1 {10}$.
 \item[Case 2] Pick generic $G_1,G_2$ of degree $3$ and let
 $$
 X_{[\lambda:\mu]} = \{ F=0 \} \cap \{\lambda G_1 + \mu G_2 =0  \}
 $$
 The fibration has $98$ singular fibers and $\deg H^{2,0}=1$, therefore $\lambda_1 = \frac 1{192}$.
\end{description}

\paragraph{Intersection of three quadrics in $\bP^5$.}
Let $F_1,F_2,F_3$ be generic homogeneous degree $2$ polynomials in $X_0,\ldots , X_5$.
The intersection of their zero loci gives a smooth K3 surface.
We let the last one vary to obtain a family
$$
X_{[\lambda:\mu]} = \{F_1=0\}\cap \{F_2=0\}\cap \{\lambda F_3 + \mu F_4 =0\}
$$
This family over $\bP^1$ has $80$ nodal fibers and $\deg H^{2,0}=1$, therefore $\lambda_1 = \frac 1 {156}$.

\paragraph{Double cover of a sextic in $\bP^2$.}
Let $F(X_0,X_1,X_2)$ be a generic homogeneous polynomial of degree $6$.
The equation
$$
y^2 = F(X)
$$
determines a double cover of $\bP^2$ ramified along the locus $\{F=0\}$.
Let $F$ vary in a pencil to get a family
$$
X_{[\lambda:\mu]} = \{ y^2 = \lambda F_1 + \mu F_2 \}
$$
There are $150$ nodal fibers and $\deg H^{2,0}=1$, therefore $\lambda_1 = \frac 1 {294}$.

\subsection{Fermat and Dwork families}

We now work out an example of a family of K3 surfaces related to mirror symmetry.
The period map and the corresponding Picard-Fuchs equations have been computed by Hartmann \cite{Hartmann}.
The thesis of Smith \cite{Smith} provides another large set of examples which can be worked out.

\paragraph{Fermat pencil.}
Consider the family of quartic surfaces $X_t\subset \bP^3$ depending on a parameter $t\in \bA^1\subset \bP^1$ given by
$$
X_0^4+X_1^4+X_2^4+X_3^4 - 4t X_0 X_1 X_2 X_3 = 0
$$
This family has singular fibers at $t=\infty$ and the roots of unity $\{\zeta\vert \zeta^4=1\}$.
Let $G = (\bmu_4)^4$ be a group with four copies of the $4$th roots of unity.
Letting each copy of $\bmu_4$ act on a homogeneous coordinate of $\bP^3$, the subgroup
$$
G_0 = \{(\lambda_0,\lambda_1,\lambda_2,\lambda_3)\in (\bmu_4)^4 \vert \lambda_0\lambda_1\lambda_2\lambda_3 = 1\}
$$
preserves each surface $X_t$.
Taking the quotient by the action of $G_0$ gives a family $Y_t\to \bP^1$ whose fibers are singular quartic surfaces.

The ring of invariants for the action of $G_0$ is generated by the five monomials
\begin{gather*}
 Y_0=X_0^4 \hskip 1em Y_1=X_1^4 \hskip 1em Y_2 =X_2^4 \hskip 1emY_3=X_3^4\\
 Y_4 = X_0 X_1 X_2 X_3
\end{gather*}
This gives an embedding of $Y_t=X_t/G_0$ into $\bP^4$ as the family
\begin{align*}
 Y_0 + Y_1 + Y_2 + Y_3 - 4t Y_4 &= 0\\
 Y_0 Y_1 Y_2 Y_3 - Y_4^4 &=0
\end{align*}
A general fiber has $6$ singularities of type $A_3$, a fiber over $t^4=1$ has an additional $A_1$ singularity, and the fiber at infinity is a union of $4$ planes.

\paragraph{The Dwork family.}
The fibers of the family $Y_t \to \bP^1$ have a simultaneous desingularization $\tilde{Y}_t\to \bP^1$ called the Dwork family.
It is obtained by a minimal desingularization of the general fibers.
The remaining singular fibers are again only over $t^4=1$ and $\infty$.

The period map of the Dwork family is computed by Hartmann \cite{Hartmann} and depicted in \autoref{img:per_map}.
The monodromy of the Dwork family acts non-trivially only on a real $3$-dimensional piece, on which the metric has signature $(2+,1-)$.
The corresponding period domain is isomorphic to the upper half-plane.

\begin{figure}[ht]
 \centering
 \includegraphics[keepaspectratio=true, width=\textwidth]{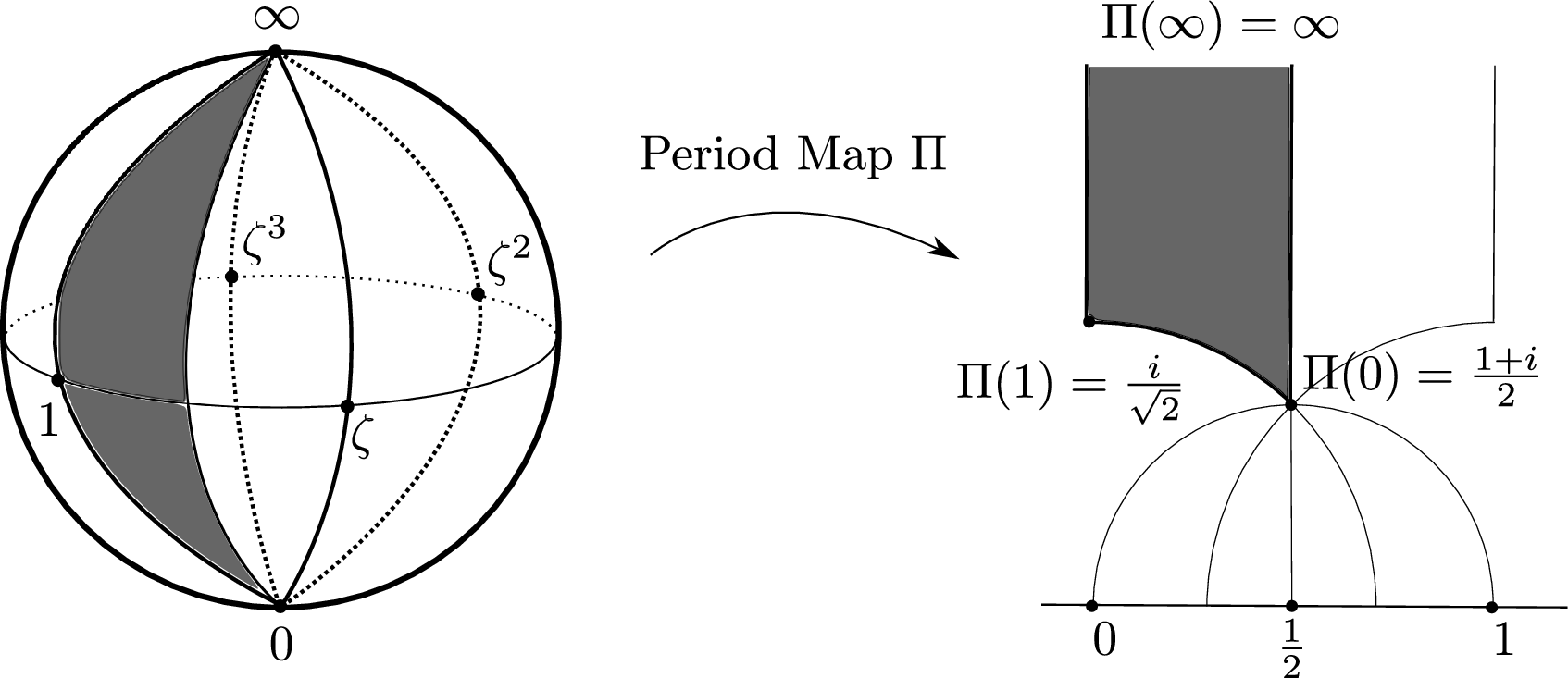}
 \caption{On the wedge of sphere on the left, the period map is a Schwarz triangle mapping.}
 \label{img:per_map}
\end{figure}
To compute the degree of the Hodge bundle, it suffices to integrate the curvature form.
In this case, because the period map is given by Schwartz triangle mappings (on appropriate charts), it suffices to sum the area of the corresponding (eight) hyperbolic triangles.
We have
$$
\int_{\bP^1\setminus \{t^4=1\} \cup \infty} \Omega_{H^{1,0}}
= 8\int_{0}^{1/2} \int_{\sqrt{1/2-x^2}}^\infty \frac {dx\, dy}{y^2} = 2\pi
$$
Since the hyperbolic area of the punctured $\bP^1$ in this case is $6\pi$, we find
$$
\lambda_1 = \frac 12 \cdot \frac{2\pi}{6\pi} = \frac 16
$$
Note that in this case $\lambda_2=0$, since the monodromy is highly reducible.

\subsection{Maximal Lyapunov exponent and Kummer K3s}
\label{subsec:max_exp}

Given a variation of Hodge structures over a punctured curve $C\setminus~S$, Arakelov inequalities give apriori bounds on the degrees of Hodge bundles (see e.g. the work of Viehweg and Zuo \cite{Viehweg}).
For variations of K3 type, we have
\begin{align}
\label{eqn:Arakelov_ineq}
 \deg H^{2,0} \leq \deg K_C(\log S)
\end{align}
According to the \autoref{thm:top_exp}, this implies the top Lyapunov exponent is bounded above by $\frac 12$.
The work of Sun, Tan, and Zuo \cite{MaxK3} describes the situations when equality occurs.

\paragraph{Kummer K3s.}
Given a torus $A$ of complex dimension two, the Kummer construction associates to it a K3 surface $K(A)$ as follows.
Let $A_0=A/\{\pm 1\}$ be the quotient of $A$ by the involution $x\mapsto -x$ on $A$.
Then $K(A)$ is defined as the blow-up of $A_0$ at the $2^4=16$ singular points of $A_0$.

This construction also works in families.
Given a family $A\to C$ of abelian surfaces over a curve $C$, let $K(A)\to C$ be the associated family of Kummer K3s.

The simplest abelian surfaces are the product of two elliptic curves.
Given two families $E_{(i)}\to C$ of elliptic curves, let $K(E_{(1)}\times E_{(2)})\to C$ be the associated family of Kummer K3s.
If $H^{p,q}_{(i)}$ denotes Hodge bundles of the family $E_{(i)}$, then the Hodge bundle $H^{2,0}$ of the Kummer family satisfies
$$
H^{2,0} = H^{1,0}_{(1)}\otimes H^{1,0}_{(2)}
$$
The isomorphism works not just at the level of bundles, but also in the cohomology of the fibers (see \cite[Prop. 4.3]{Morrison}).

Note that the monodromy of the family has a large constant part (at least coming from the blow-up divisors, but in general more).
Denote by $T_{K(E_{(1)}\times E_{(2)})}$ the irreducible part (also called the ``transcendental part").
If the elliptic curves in $E_{(1)}$ and $E_{(2)}$ are isogenous over each point in $C\setminus S$, then the transcendental part is $3$-dimensional.

\paragraph{Maximal Lyapunov exponent.}
Suppose now that $H=H^{2,0}\oplus H^{1,1}\oplus H^{0,2}$ is an irreducible variation of Hodge structure of K3 type over a Riemann surface $C\setminus S$.
Suppose that its top Lyapunov exponent is $\frac 12$, i.e. we have equality in \autoref{eqn:Arakelov_ineq}.
Theorem 0.1 and Lemma 1.1 from \cite{MaxK3} describe this situation.

Equality occurs if and only if $H$ is the transcendental part of a family of Kummer K3s coming from a product of elliptic curves.
Moreover, the families of elliptic curves $E_{(i)}$ must be modular (i.e. pulled back via covering maps $C\setminus S \to \bH/\SL_2(\bZ)$).
In particular, the fibers over any given point of $C\setminus S$ are isogenous.

%====================================================
%====================================================
%		End of Main Article
%====================================================
%====================================================

%====================================================
%====================================================
%		Appendix
%====================================================
%====================================================

\appendix
%\section{}

%		End of Appendix
%====================================================

%====================================================
%		Bibliography
%====================================================
\bibliographystyle{sfilip}
\bibliography{K3_Lyapunov}
%====================================================
\end{document}